\documentclass[a4paper,12pt]{amsart}
\usepackage{amsmath,amsthm,amssymb}
\usepackage{hyperref}

\allowdisplaybreaks[1]

\numberwithin{equation}{section}
\newtheorem{theorem}{Theorem}[section]

\newtheorem{proposition}[theorem]{Proposition}
\newtheorem{corollary}[theorem]{Corollary}
\theoremstyle{remark}
\newtheorem{remark}[theorem]{Remark}
\newtheorem{example}{Example}[section]
\newtheorem{definition}{Definition}[section]

\newcommand{\Dl}{\Delta}
\newcommand{\dl}{\delta}
\newcommand{\e}{\epsilon}
\newcommand{\w}{\omega}
\newcommand{\Om}{\Omega}
\newcommand{\ity}{\infty}

\newcommand{\R}{\mathbb{R}}
\newcommand{\Ob}{\mathcal{O}}
\newcommand{\ti}{\widetilde}

\newcommand{\N}{\mathbb{N}}
\newcommand{\Z}{\mathbb{Z}}
\newcommand{\al}{\alpha}

\begin{document}

\title[Recurrence in generalized semigroup]
{Recurrence in generalized semigroup}

\author[K. Lalwani]{Kushal Lalwani}
\address{Kushal Lalwani\\Department of Mathematics\\ University of Delhi\\Delhi--110 007, India}
\email{lalwani.kushal@gmail.com }

\thanks{This work was supported by research fellowship from University Grants Commission (UGC), New Delhi.}

\subjclass[2010]{54H20, 37B20, 54H15}
\keywords{chain recurrent set, nonwandering set, omega limit set;  recurrence}

\begin{abstract}
In \cite {kl}, we introduced the concept of escaping set in general setting for a topological space and extended the notion of limit set and escaping set for the general semigroup generated by continuous self maps. In this paper we continue with extending the other notions of recurrence for the generalized semigroup analogs to their counterpart in the classical  theory of dynamics. We discuss the concept of periodic point, nonwandering point and chain recurrent point in the more  general setting of a semigroup and establish the correlation between them. 
\end{abstract}

\maketitle

\section{Introduction}

This paper is in continuation with our earlier paper \cite {kl} where we extended the classical concepts of $\w$-limit points and recurrent points of abstract dynamics  to the dynamics of an arbitrary semigroup. Hinkkanen and Martin \cite{hm} pioneered the notion of semigroup to extend the theory of complex dynamics and in fact, this work propelled us towards this study. Our aim here is to see how far the various notions of recurrence applies in this general setting.

The classical concept of a {\it periodic point} is based on the group property of integers or real numbers with addition, where a point is self traced under the dynamics of a map at some regular `period\rq{}. In the general setting of a semigroup one can not expect such a nice behavior. However,  we have extended the concept following the nature of a point to be self traced  under the dynamics of some map. We have termed such a point as an {\it orbit point}. 

Further we extend the notion of a {\it nonwandering point} and {\it chain recurrent  point}. The situation in the general setting is not at all straightforward. In contrast to the classical theory, a recurrent point need not be nonwandering and a nonwandering point need not be chain recurrent. For this, we have strengthen the notion of recurrent point and nonwandering point. We have established the notion of a {\it strongly nonwandering point}, though it is equivalent with the nonwandering point in the abstract dynamics. We will show that a strongly nonwandering point is chain recurrent. Also, we will show that an orbit point need not be chain recurrent and hence not a strongly nonwandering point.

In the classical theory, a topological dynamical system or a flow is a triple $(X,T,\Phi)$ consisting of a topological space $X$ (the phase space) with a group (or a semigroup) $T$ (the phase group) and a continuous mapping $\Phi :X \times T \to X$ such that\\
(i) $\Phi^t(.)=\Phi(.,t): X \to X$ is continuous for each $t\in T$, and\\
(ii) $\Phi(x,ts)=\Phi(\Phi(x,t),s)$ for all $t,s\in T$.

Generally, $T$ is either the additive group of reals $(\R,+)$ or integers $(\Z,+)$ or the set of positive integers $(\N,+)$. Also if $T=\Z$ or $\R$, the map $\Phi^t$ is invertible (hence a homeomorphism) for each $t$ and $(\Phi^t)^{-1}=\Phi^{-t}$. In that case $\Phi^0$ is the identity map on $X$. 

A {\it continuous semigroup}  is a set of (non-identity) continuous self maps, of a topological space $X$,  which are closed under the composition. A semigroup $G$ is said to be generated by a family $\{g_{\al}\}_{\al}$ of  continuous self maps of a topological space $X$  if every element of $G$ can be expressed as compositions of iterations of the elements of  $\{g_{\al}\}_{\al}$. We denote this by $G=<g_{\al}>_{\al}$. The space $X$ is assumed to be Hausdorff and first countable.

Given a semigroup $G$ and $x \in X$, the set $\Ob_G(x):=\{g(x): g \in G\}\cup \{x\}$ is called the  {\it orbit} of $x$ under $G$.
A subspace $Y \subset X$ is said to be  {\it{invariant}} under $G$ if $g(y)\in Y$ for all $g\in G$ and $y\in Y$.
If $G=\ <g_{\al}>_{\al}$ then for $Y$ to be invariant it is sufficient that $g(y)\in Y$ for all $g\in \{g_{\al}\}_{\al}$ and $y\in Y$.

In \cite {kl}, we have introduced the notion of topological conjugacy of dynamics of two semigroups on two topological spaces as follows:
Let $G=\ <g_i>_{i \in \varLambda}$ be a semigroup of continuous self maps of $X$ and $\ti{G}=\ <\ti{g_i}>_{i \in \varLambda}$ be a semigroup of continuous self maps of $Y$. Two dynamical systems $(X,G)$ and $(Y,\ti{G})$ are said to be topologically conjugate if there exists a homeomorphism $\rho : X \to Y$ such that $\rho\circ g_i=\ti{g_i}\circ \rho$ for each $i \in \varLambda$.
Note that for $g \in G$, we have, $g=g_{i_1}\circ \ldots \circ g_{i_n}$ for some $g_{i_j}\in \{g_i\}_{i \in \varLambda}$.  If $\rho : X \to Y$ is a topological conjugacy then
\begin{equation} \notag
\begin{split}
\rho \circ g &= \rho \circ g_{i_1}\circ \ldots \circ g_{i_n}\\
&=\ti{g}_{i_1}\circ \ldots \circ \ti{g}_{i_n}\circ \rho\\
&=\ti{g} \circ \rho
\end{split}
\end{equation}
for $\ti{g} \in \ti{G}$.

For general reference to standard terms and basic facts from the classical theory of dynamics we follow  Alongi and Nelson\rq{}s book \cite{Alongi}.

\section{Orbit Points}

In case of a flow $(X,T,\Phi)$, a point $p \in X$ is called a {\it fixed point} if $\Phi^t(p)=p$ for all $t \in T$. Analogous to this, in the theory of transformation groups: a point $p$ of a $G$-space $X$ is called a {\it fixed point} if $g(p)=p$ for all $g \in G$. Also $p \in X$ is called a {\it periodic point} of the flow if $\Phi^t(p)=p$ for some positive $t \in T$. Here we are defining a new characteristic of a point which is analogous to the periodic point in case of a generalized semigroup.

\begin{definition}
Let $G$ be a semigroup on a topological space $X$. A point $p\in X$ is called an {\it orbit  point} for $G$ if $g(p)=p$ for some $g \in G$.The set of all orbit points for $G$ is denoted by $Orb(G)$.
\end{definition}
Note that the term {\it orbit point} is a characteristic of a point associated with its  nature: to be self traced  under the dynamics of some map of the semigroup. Whereas, the {\it orbit of a point} is a set consisting of all the points traced by it under the semigroup including itself. Therefore, one must be careful while dealing with the two terminologies. 

\begin{theorem}
If $G$ is abelian then $Orb(G)$ is an invariant subset of $X$.
\end{theorem}

\begin{proof}
Let $G$ be an abelian semigroup and $p \in Orb(G)$. Since $p$ is an orbit point, there is some $h \in G$ such that $h(p)=p$. Let $g \in \{g_{\al}\}_{\al}$ be any generator of $G$. Then
$$h(g(p))=g(h(p))=g(p).$$
Thus $g(p)\in Orb(G)$ for all $g\in \{g_{\al}\}_{\al}$. Hence $Orb(G)$ is invariant under $G$.
\end{proof}

\begin{theorem}
Let $(X,G)$ and $(Y,\ti{G})$ be two dynamical systems. If $\rho : X \to Y$ is a topological conjugacy then $\rho(Orb(G))=Orb(\ti{G})$.
\end{theorem}

\begin{proof}
Let $p \in Orb(G)$, that is, $g(p)=p$ for some $g \in G$. Since $\rho : X \to Y$ is a topological conjugacy, we have
$$\ti{g}\circ \rho(p)=\rho \circ g(p)=\rho(p).$$
Thus $\rho(p) \in Orb(\ti G)$ and hence $\rho(Orb(G)) \subset Orb(\ti{G})$.

Conversely, let $p \in Orb(\ti{G})$. Since $\rho^{-1}$ is a conjugacy from $Y$ to $X$, we have, $p \in \rho(Orb(G))$. Thus, $\rho(Orb(G))=Orb(\ti{G})$.
\end{proof}

The following example shows that the orbit of an {\it orbit point} is not closed in general. Whereas the orbit of a {\it periodic point } is closed see \cite {Alongi}. Also one can see that in the following non-abelian semigroup the {\it orbit point} is not invariant.

\begin{example}
Let $X =\R$ and $G=<g_1,g_2>$, where $g_1=x^2$ and $g_2=x^2-\frac{1}{2}$. Then $p=1$ is an {\it orbit point} of $(X,G)$. Also the sequence $<g_1(p),g_2g_1(p),g_1g_2g_1(p),g_1^2g_2g_1(p),g_1^3g_2g_1(p),$ $\ldots ,g_1^n g_2g_1(p), \ldots>=<1,1/2,1/4,1/8,\ldots 1/2^n,\ldots>$ is from $\Ob(p)$ and tends to $0$. But $0 \notin \Ob(p)$.
\end{example}

In \cite {kl} we have introduced the notion of $\w$-limit point and recurrent point for the generalized semigroup as follows:

\begin{definition} \label{ub}
A sequence of functions $(f_{n_k})$ in $G$ is said to be {\it{unbounded}} if  $n_k\to \ity$ as $k\to \ity$ and each $f_{n_k}$ consists of exactly $n_k$ iterates of  $g_{\al_{0}}$, for fix $g_{\al_0}\in \{g_{\al}\}_{\al}$, that is, $f_{n_k}=h_1\circ g_{\al_0}\circ h_2 \circ g_{\al_0} \circ h_3 \circ \ldots \circ h_{n_k}\circ g_{\al_0} \circ h_{n_k+1}$, where each $h_i \in G\cup \{identity\}$ and the functions $h_i$ are independent of $g_{\al_0}$.
\end{definition}

\begin{definition}
A point $z\in X$ is called an {\it{$\w-$limit\ point}} for a point $x\in X$ if for some unbounded sequence $(f_{n_k})$ in $ G, \ f_{n_k}(x)\to z$ as $k \to \ity$. The $\w$-limit set $\w(x)$, is the set of all  $\w$-limit points for $x$ .
\end{definition}

\begin{definition}
Let $(X,G)$  be a topological dynamical system. A point $x\in X$ is called a {\it{recurrent point}} with respect to $G$ if  $x\in \w(x)$. The set of all recurrent points of $(X,G)$ is denoted by $Rec(G)$.
\end{definition}

As in the classical case of a flow, one can see the following relationship among various subsets of $X$ for the generalized semigroup,
$$Fix(G) \subset Orb(G) \subset Rec(G).$$
Here, $Fix(G)$ denotes the set of fixed point for the generalized semigroup, as defined above.

\section{Nonwandering Points}

Next we shall define the concept of nonwandering points for the generalized semigroup.

\begin{definition} \label{nw}
Let $G$ be a semigroup on a topological space $X$. A point $x\in X$ is called a {\it nonwandering  point} for $G$ if for each neighborhood $U$ of $x$, $g(U)\cap U \neq \emptyset$ for some $g \in G^*=G \setminus \{g_{\al}\}_{\al}$. In the contrary case $x$ will be called a {\it wandering point} for $G$. The set of all nonwandering points for $G$ is denoted by $\Om(G)$.
\end{definition}

\begin{theorem}
If $G$ is abelian then the set of all nonwandering points is an invariant subset of $X$.
\end{theorem}

\begin{proof}
Let $x \in \Om(G) $ and $g \in \{g_{\al}\}_{\al}$. Let $U$ be a neighborhood of $g(x)$. Since $g$ is continuous, we have, $V=g^{-1}(U)$ a neighborhood of $x$. 

Since  $x \in \Om(G) $, there exists $\hat{g} \in G^*$ such that $V \cap \hat{g}{V} \neq \emptyset$. This gives $g(V) \cap g\hat{g}(V)\neq \emptyset$. Since $G$ is abelian, we have, $U \cap \hat{g}(U) \neq \emptyset$. Thus $g(x) \in \Om(G) $ and hence $\Om(G)$ is an invariant subset of $X$.
\end{proof}

\begin{theorem}
The set of all nonwandering points  is a closed subset of $X$.
\end{theorem}

\begin{proof}
Let $\hat{x} \in \overline{\Om(G)}$ and $U$ be a neighborhood of $x$.

Since $\hat{x} \in \overline{\Om(G)}$, there exists some $x \in U \cap \Om(G)$. Then $x$ is a nonwandering point and this implies $g(U)\cap U \neq \emptyset$ for some $g \in G^*$. Thus  $\hat{x} \in \Om(G) $ and hence $\Om(G)$ is a closed subset of $X$.
\end{proof}

\begin{theorem}
Let $(X,G)$ and $(Y,\ti{G})$ be two dynamical systems. If $\rho : X \to Y$ is a topological conjugacy then $\rho(\Om(G))=\Om(\ti{G})$.
\end{theorem}

\begin{proof}
Let $x \in \Om(G) $. Let $U$ be a neighborhood of $\rho(x)$. Since $\rho$ is continuous, we have, $V=\rho^{-1}(U)$ a neighborhood of $x$. 

Since  $x \in \Om(G) $, there exists $g \in G^*$ such that $V \cap g{V} \neq \emptyset$. This gives $\rho(V) \cap \rho g(V)\neq \emptyset$. Since $\rho$ is a conjugacy, we have, $U \cap \ti{g}(U) \neq \emptyset$. Thus $\rho(x) \in \Om(\ti G) $ and hence $\rho(\Om(G)) \subset \Om(\ti{G})$.

Conversely, let $y \in \Om(\ti{G})$. Since $\rho^{-1}$ is a conjugacy from $Y$ to $X$, we have, $y \in \rho(\Om(G))$. Thus, $\rho(\Om(G))=\Om(\ti{G})$.
\end{proof}

However, in definition \ref{ub} if we assume, in addition, that the unbounded sequence satisfies the relation:
$$f_{n_k+1}=g \circ f_{n_k},$$
for each $k \in \N$  and for some $g\in G$. Then one can conclude the following proposition.

\begin{proposition}
The set of all nonwandering points contains the $\w$-limit set, obtained with the above additional assumption, for all $x \in X$.
\end{proposition}

\begin{proof}
Let $y \in \w(x)$ and $U$ be a neighborhood of $y$ in $X$. Since $y \in \w(x)$,  there exists an unbounded sequence $(f_{n_k})$ in $G$ such that $f_{n_k} \to y$ as $k \to \ity$. Then there exists some $N\in \N$ such that $f_{n_j}(x)\in U$ for each $j \geq N$. Since $n_k \to \ity$ as $k\to  \ity $, there exists some $m > N$ such that $f_{n_m}=g \circ f_{n_N}$ for some $g \in G^* $. Therefore, $f_{n_m}(x)=g \circ f_{n_N}(x)\in g(U)$.

Thus, $f_{n_m}(x) \in U \cap g(U)$ and hence $y \in \Om(G)$.
\end{proof}

\begin{corollary}
The set of all nonwandering points contains the set of recurrent points, obtained with the above additional assumption.
\end{corollary}

\begin{definition} \label{snw}
Let $G$ be a semigroup on a topological space $X$. A point $x\in X$ is called a {\it strongly nonwandering  point} for $G$ if for each neighborhood $U$ of $x$ and each $g \in G$, $h\circ g(U)\cap U \neq \emptyset$ for some $h \in \widehat G=G \cup \{identity\}$. The set of all strongly nonwandering points for $G$ is denoted by $\Om_S(G)$.
\end{definition}

However, in the classical case of a continuous or discrete flow the definitions \ref{nw} and \ref{snw} are equivalent,  see \cite {Alongi}. Clearly, every strongly nonwandering point is a nonwandering point. But in general the other implication does not holds. This can be seen from the following example.

\begin{example} \label{ex}
Let $X =\R$ and $G=<g_1,g_2>$, where $g_1=x^2$ and $g_2=x^3$. Then $p=-1$ is an {\it orbit point} of $(X,G)$ and hence a nonwandering point. Now for the neighborhood $U=(-\ity,0)$ of $p$, $g_1(U)=(0,\ity)$ is an invariant subset of $X$ under $G$. Therefore, $h\circ g_1(U)\cap U = \emptyset$ for any $h \in \widehat G$. Hence $p$ is not a strongly nonwandering point of $G$.
\end{example}

The proofs of  following theorems follow with the similar arguments as in the previous case of nonwandering points.

\begin{theorem}
If $G$ is abelian then the set of all strongly nonwandering points is an invariant subset of $X$.
\end{theorem}

\begin{theorem}
The set of all strongly nonwandering points for $G$  is a closed subset of $X$.
\end{theorem}

\begin{theorem}
Let $(X,G)$ and $(Y,\ti{G})$ be two dynamical systems. If $\rho : X \to Y$ is a topological conjugacy then $\rho(\Om_S(G))=\Om_S(\ti{G})$.
\end{theorem}

\section{Chain Recurrent Points}

\begin{definition}
Let $G$ be a semigroup on a metric space $(X,d)$. Let $a,b \in X$, $g\in G$ and $\e >0$ be  given. An $(\e,g)${\it-chain} from $a$ to $b$ means a finite sequence $(a=x_1,\ldots, x_{n+1}=b;g_1, \ldots, g_n)$, where for every $i, \ x_i\in X$ and $g_i \in \widehat G$ such that $d(g_ig(x_i),x_{i+1})< \e$ for $i=1 \ldots n$.
\end{definition}

\begin{definition}
Let $G$ be a semigroup on a metric space $(X,d)$. A pair of points $a,b\in X$ are called  {\it chain equivalent points} for $G$ if for every $\e>0$ and $g\in G$ there exists an $(\e, g)$-chain from $a$ to $b$ and  an $(\e, g)$-chain from $b$ to $a$. The set of all chain equivalent points for $G$ is denoted by $CE(G)$.
\end{definition}

\begin{definition}
Let $G$ be a semigroup on a metric space $(X,d)$. A point $x\in X$ is called a {\it chain recurrent point} for $G$ if for every $\e>0$ and $g\in G$ there exists an $(\e, g)$-chain from $x$ to itself. The set of all chain recurrent points for $G$ is denoted by $CR(G)$.
\end{definition}

\begin{theorem} \label{ce}
The set of all chain equivalent  points for $G$  is closed.
\end{theorem}

\begin{proof}
 Suppose that $(a,b)$ is a limit point of $CE(G)$. Let $\e >0$ be given and some $g \in G$. First we shall construct an $(\e,g)$-chain from $a$ to $b$.

Since $g$ is continuous, there are $\dl_1,\dl_2>0$  such that if
$$d(x,a)<\dl_1 \ {\rm then} \ d(g(x),g(a))< \e$$
and
$$d(y,b)<\dl_2 \ {\rm then} \ d(g(y),g(b))< \e.$$

Take $\dl=min\{\dl_1,\dl_2,\e/2\}$. Since $(a,b)$ is a limit point of $CE(G)$, there exists a pair of chain equivalent point $(x,y)$ such that $d(x,a)<\dl$ and $d(y,b)<\dl$. Then $d(g(x),g(a))< \e$ and hence $(a,g(x);identity)$ is an $(\e,g)$-chain from $a$ to $g(x)$.

Since $(x,y)\in CE(G)$, there exists an $(\e/2,g^2)$-chain $(x=x_1,\ldots, x_{n+1}=y;g_1, \ldots, g_n)$ from $x$ to $y$. Then $(g(x),x_2\ldots, x_n;g_1,g_1g, \ldots, g_{n-1}g)$ is an $(\e,g)$-chain from $g(x)$ to $x_n$. 

Since
$$d(g_ng^2(x_n),b)<d(g_ng^2(x_n),y)+d(y,b)<\frac{\e}{2}+\frac{\e}{2}=\e,$$
we have, $(x_n,b;g_ng)$ is an $(\e,g)$-chain from $x_n$ to $b$.

By transitivity via concatinating these $(\e,g)$-chains, there is a $(\e,g)$-chain from $a$ to $b$.

Also interchanging the role of $a$ and $b$; and $x$ and $y$, on the similar line we can produce a $(\e,g)$-chain from $b$ to $a$. Therefore, the set of all chain equivalent  points for $G$  is closed.
\end{proof}

\begin{theorem}
The set of all chain recurrent  points for $G$  is a closed subset of $X$.
\end{theorem}

\begin{proof}
Since $X$ is a metric space the {\it diagonal} $\Dl=\{(x,x):x\in X\}$ is closed in $X\times X$. Also by theorem \ref {ce}, the set $CE(G)$ is closed, therefore, the set $\Dl \cap CE(G)$ is a closed subset of $\Dl$. Since the canonical projection map $\pi_1 : X \times X \to X$ restricted to $\Dl$ is a homeomorphism, it follows that the set $CR(G)=\pi_1|_{\Dl}(\Dl \cap CE(G))$  is a closed subset of $X$.
\end{proof}

\begin{theorem}
If $G$ is abelian then the chain recurrent set $CR(G)$ is invariant under $G$.
\end{theorem}

\begin{proof}
Let $x \in CR(G)$ and $g$ be any generator of $G$.  Let $\e >0$ be given and some $h \in G$. We shall construct an $(\e,h)$-chain from $g(x)$ to itself and it will follow that $g(x)\in CR(G)$.

Since $g$ is continuous, there exists a $0<\dl \leq \e$ such that if $d(x,y)< \dl \ {\rm then}\  d(g(x),g(y)) < \e.$  Also $x \in CR(G)$ implies there is a $(\dl, g\circ h)$-chain $(x=x_1,\ldots , x_{n+1}=x;h_1, \ldots ,h_n)$ from $x$ to itself. That is, for each $i=1, \ldots ,n$, we have,
$$d(h_i \circ g \circ h(x_i),x_{i+1})<\dl \leq \e.$$

In particular, since $G$ is abelian, we have,
$$d(h_1 \circ h \circ g(x_1),x_2)=d(h_1 \circ g \circ h(x_1),x_2)<\dl \leq \e.$$

Also by the continuity of $g$, we have,
$$d(h_n \circ g \circ h(x_n),x)<\dl \ {\rm implies} \ d(g \circ h_n \circ g \circ h(x_n),g(x))<\e.$$

Hence $(g(x),x_2,\ldots , x_n,g(x);h_1,h_2 \circ g, \ldots,h_{n-1}\circ g ,g \circ h_n\circ g)$ is  an $(\e,h)$-chain from $g(x)$ to itself.
\end{proof}

\begin{theorem}
Let $(X,G)$ and $(Y,\ti{G})$ be two dynamical systems on the metric spaces $(X,d_X)$ and $(Y,d_Y)$. If a uniform homeomorphism $\rho : X \to Y$ is a topological conjugacy then $\rho(CR(G))=CR(\ti{G})$.
\end{theorem}

\begin{proof}
Let $x \in CR(G)$.  Let $\e >0$ be given and some $\ti g \in \ti G$. First we shall construct an $(\e,\ti g)$-chain from $y=\rho(x)$ to itself. Let $g \in G$ be such that $\rho \circ g=\ti g \circ \rho$.

Since $\rho : X \to Y$ is uniformly continuous, there exists a $\dl >0$ such that if $d_X(x_1,x_2)< \dl$ then $d_Y(\rho(x_1),\rho(x_2))< \e$, for $x_1,x_2 \in X$.  Also $x \in CR(G)$ implies there is a $(\dl, g)$-chain $(x=x_1,\ldots , x_{n+1}=x;g_1, \ldots ,g_n)$ from $x$ to itself. That is, for each $i=1, \ldots ,n$, we have,
$$d_X(g_i\circ g(x_i),x_{i+1})<\dl .$$

For  each $i=1, \ldots ,n+1$ let $y_i=\rho(x_i)$. Since $d_X(g_i\circ g(x_i),x_{i+1})<\dl $, we have,
$$ d_Y(\ti g_i\circ \ti g(y_i),y_{i+1})=d_Y(\rho \circ g_i\circ g(x_i),\rho(x_{i+1})) <\e.$$

Thus $(\rho(x)=y_1,\ldots , y_{n+1}=\rho(x); \ti g_1, \ldots ,\ti g_n)$ from $\rho(x)$ to itself and hence $\rho(x)\in \rho(CR(G))$.  Therefore $\rho(CR(G)) \subset CR(\ti{G})$.

Conversely, let $y \in CR(\ti{G})$. Since $\rho^{-1}$ is a conjugacy from $Y$ to $X$, we have, $y \in \rho(CR(G))$. Thus, $\rho(CR(G))=CR(\ti{G})$.
\end{proof}

\begin{remark}
If $G$ consists of uniformly continuous maps then the chain recurrent set $CR(G)$ is invariant under $G$ without the assumption that $G$ is abelian.
\end{remark}

\begin{proposition}
The set of chain recurrent points on a metric space contains the set of strongly nonwandering points.
\end{proposition}

\begin{proof}
Let $(X,d)$ be a metric space and $x \in \Om_S(G)$. Let $g \in G$ and $\e >0$ be given. Since $g$ is continuous, there exists a $0 < \dl < \e$ such that $d(g(x),g(y))<\e$ whenever $d(x,y)<\dl$.

Since $x \in \Om_S(G)$, there exists some $h \in \widehat G$ such that $h \circ g^2(B(x,\dl))\cap B(x,\dl) \neq \emptyset$, where $B(x,\dl):=\{y:d(x,y)<\dl\}$. Therefore there exists $y \in  B(x,\dl)$ such that $ h \circ g^2(y) \in B(x,\dl)$. Since $d(x,y)<\dl$ we have $d(g(x),g(y))<\e$. Also  $ h \circ g^2(y) \in B(x,\dl)$ implies $d(hg^2(y),x)<\dl <\e$.

Thus, $(x,g(y),x;identity,h)$ is an $(\e,g)$-chain from $x$ to itself and hence $x \in CR(G)$.
\end{proof}

\begin{remark}
As compared  to the classical theory of dynamics, the orbit point and hence a nonwandering point need not be chain recurrent in the general setting of a semigroup.  In example \ref{ex}, one can see that  for $\e <1$ there does not exist any $(\e,g_1)$-chain from $p=-1$ to itself.
\end{remark}

{\bf Acknowledgment.}
I am thankful to my thesis adviser Sanjay Kumar for fruitful discussions.

\end{document}